\documentclass[12pt]{amsart}
\usepackage{amsmath,amsfonts,latexsym,graphicx,amssymb,url}
\usepackage{hyperref,pdfsync}
\usepackage{amsmath,txfonts,pifont,bbding,pxfonts,manfnt}
\usepackage[active]{srcltx}
\usepackage{wasysym,pstricks}
\usepackage{lscape,color}

\numberwithin{equation}{section}
\newcommand{\R}{{\mathbb R}} %%reals

\newcommand{\p}{\varphi}
\renewcommand{\(}{\left(}
\renewcommand{\)}{\right)}
\newtheorem{theorem}{Theorem}[section]

\newtheorem{lemma}[theorem]{Lemma}

\newtheorem{definition}[theorem]{Definition}
\newtheorem{remark}[theorem]{Remark}

\begin{document}

%\setcounter{equation}{0}

%\documentstyle[12pt]{article}

%\usepackage{amsmath}
%\usepackage{amsfonts}
%\usepackage{amssymb}
%\usepackage{amscd}

%\setlength{\headheight}{15pt}
%\setlength{\topmargin}{10pt}
%\setlength{\headsep}{30pt}
%\setlength{\textwidth}{15cm}
%\setlength{\textheight}{21.5cm}
%\setlength{\oddsidemargin}{1cm}
%\setlength{evensidemargin}{1cm}

%\newtheorem{theorem}{Theorem}[section]

%\newtheorem{lemma}[theorem]{Lemma}

%\newtheorem{corollary}[theorem]{Corollary}

%\newtheorem{proposition}[theorem]{Proposition}

%\newtheorem{definition}[theorem]{Definition}

%\begin{document}

\title[Refractors in anisotropic media]{Refractors in anisotropic media\\associated with norms}
\author[C. E. Guti\'errez,  Qingbo Huang and H. Mawi]{Cristian E. Guti\'errez,  Qingbo Huang and Henok Mawi}
\begin{abstract}
We show existence of interfaces between two anisotropic materials so that light is refracted in accordance with a given pattern of energy. To do this we formulate a vector Snell law for anisotropic media when the wave fronts are given by norms for which the corresponding unit spheres are strictly convex.
\end{abstract}
\thanks{\today. \\The first author was partially supported by NSF grant DMS--1600578, and the third author partially supported by NSF grant HRD--1700236.}
\address{Department of Mathematics\\Temple University\\Philadelphia, PA 19122}
\email{gutierre@temple.edu}
\address{Department of Mathematics \& Statistics, Wright State University, Dayton, OH 45435}
\email{qingbo.huang@wright.edu}
\address{Department of Mathematics, Howard University, Washington, D.C. 20059}
\email{henok.mawi@howard.edu}

%\email{gutierre@temple.edu}

%\begin{document}
\maketitle
\tableofcontents

\section{Introduction}
Anisotropic materials are those whose optical properties vary according to the direction of propagation of light. Typical examples are crystals, where the refractive index depends on the direction of the incident light, see \cite[Chapter XV]{book:born-wolf}, \cite[Chapter XI]{landau-lifshitz:electrodynamics} and \cite[Chapter IV]{sommerfeld:optics}.
Important research was done on this subject because of it multiple applications, see the fundamental work
\cite{kline-kay:electromagneticwaves}, and \cite{yariv1988:optical-waves-in-crystals}, \cite{scharf:polarized-light-in-crystals} for more recent applications and references.
Mathematically, in these materials
wave fronts satisfy the Fresnel partial differential equation which in the particular case of isotropic materials is the eikonal equation.
A difficulty with anisotropic materials in the geometrical optics regime is that incident rays may be refracted into two rays: an ordinary ray and an extraordinary ray. This is the phenomenon of bi-refringence, observed experimentally in crystals, and is a consequence from the fact that the Fresnel equation splits in general as the product of two surfaces, see \eqref{eq:fresnel equations split}.

The main purpose of this paper is to show existence of interfaces between two homogenous and anisotropic materials so that light is refracted in accordance with a given pattern of energy.
As a main step to achieve this, we give a formulation of a vector Snell's law in anisotropic materials when the wave fronts are given by norms in $\R^n$\footnote{There is no mathematical objection to look at this problem in $n$ dimensions but the physical problem is three dimensional.} which has independent interest. 
More precisely, suppose $N_i(x)$, $i=1,2$, are norms in $\R^n$, $\Sigma_i=\{x:N_i(x)=1\}$, $\Omega_i\subset \Sigma_i$ are domains, $f>0$ is an integrable function on $\Omega_1$, and $\mu$ is a Radon measure in $\Omega_2$ with $\int_{\Omega_1}f(x)\,dx=\mu(\Omega_2)$. 
We have two anisotropic media $I$ and $II$ such that the wave fronts in $I$ are given by $N_1$ and the wave fronts in $II$ given by $N_2$. 
Light rays emanate from the origin, located in medium $I$, with intensity $f(x)$ for each $x\in \Omega_1$. We seek a surface $\mathcal S$ separating media $I$ and $II$ so that all rays emanating from the origin and with directions in $\Omega_1$ are refracted into rays with directions in $\Omega_2$ and the conservation of energy condition 
\[
\int_{\tau(E)}f(x)\,dx=\mu(E)
\] 
holds for each Borel set $E\subset \Omega_2$ where $\tau(E)=\{x\in \Omega_1:\text{$x$ is refracted into $E$}\}$.
This is called the refractor problem, and when media $I$ and $II$ are homogeneous and isotropic it is solved in \cite{gutierrez-huang:farfieldrefractor} using optimal mass transport and in \cite{gutierrez:cimelectures} with a different method.
A main difficulty to solve this problem is to lay down the mathematical formulation of the physical laws and constraints in anisotropic media to our setting with norms. 

%{\color{red}  To place our results in perspective we mention that several variants of geometric optics problems of these type that deal with reflection and refraction in homogeneous and isotropic media have been studied in X,Y,Z (put some references). In anisotropic media; a similar problem for reflection, where wave fronts are modeled by non Euclidean norm, was studied in \cite{caffarelli-huang:nonisotropicreflector}. In this paper we develop a model for refraction in anisotropic media. To do so we first state and prove the Snell's law of refraction for anisotropic media. We then use the abstract method obtained for the existence of near field refractor in homogenous and isotropic media in \cite{gutierrez-huang:nearfieldrefractor} to prove our existence results. }

To place our results in perspective we mention the following. A similar problem for norms but for reflection was studied in \cite{caffarelli-huang:nonisotropicreflector}.
Once the Snell law and physical constraints for anisotropic media are formulated and proved in Section \ref{sec:vector snell law for anisotropic media}, our existence results use the abstract method developed in \cite{gutierrez-huang:nearfieldrefractor}, where existence results for the near field refractor problem in homogenous and isotropic media are obtained. Further results on geometric optics problems for refraction in homogeneous and isotropic media have been studied in  \cite{gutierrez-mawi:refractorwithlossofenergy}, \cite{deleo-gutierrez-mawi:numericalrefractor}, and \cite{aram:parallelbeam}.
On the other hand, the mathematical literature for these problems in anisotropic media is lacking.

The paper is organized as follows. Section \ref{sec:preliminaries on norms and convexity} recalls a few results on norms and convexity that will be used later.
The Snell law in anisotropic media is proved in Section \ref{sec:vector snell law for anisotropic media} as a consequence of Fermat's principle of least time. The discussion on the physical constraints for refraction in anisotropic media is in Section \ref{subsect:physical constraints}.
Section \ref{sec:Uniformly refracting surfaces} introduces and analyzes the surfaces refracting all rays into a fixed direction which are used in Section \ref{sec:kappa less than one} to show the existence Theorems   
\ref{thm:discrete case} and \ref{thm:existence general measure}.
Section \ref{sec:propagation light in anistropic materials} introduces and analyzes the Fresnel pde for the wave fronts in general materials non homogenous and anisotropic.
In Section \ref{subsec:application of results when mu=a epsilon} we apply the results from the previous sections to materials having permittivity and permeability coefficients $\epsilon$ and $\mu$ that are constant matrices with one a constant multiple of the other.
Finally, in Section \ref{sec:optimal mass transport} we relate our problem with optimal mass transport.

\setcounter{equation}{0}

\section{Preliminaries on Norms and Convexity}\label{sec:preliminaries on norms and convexity}
\setcounter{equation}{0}
Consider a norm $N(x)=\|x\|$ in $\R^n$, and let $\Sigma=\{x\in \R^n:N(x)=1\}$ be the unit sphere in the norm such that $\Sigma$ is a strictly convex surface.
%, i.e., the set $B=\{x\in \R^n:N(x)\leq 1\}$ is strictly convex.
%The support function of $\Sigma$ is defined as follows. 

Given a vector $\nu\in S^{n-1}$, the support function of $\Sigma$ is defined by 
\[
\varphi(\nu)=\sup_{x\in \Sigma}x\cdot \nu.
\]
Clearly $\varphi$ is strictly positive ($\varphi(\nu)\geq 1/\|\nu\|$).  Since $\Sigma$ is compact, there is $x_0\in \Sigma$ such that $\varphi(\nu)=x_0\cdot \nu$, and since $\Sigma$ is strictly convex $x_0$ is unique.
The hyperplane $\Pi_\nu$ with equation $\{x:\ell(x)=x\cdot \nu-\varphi(\nu)=0\}$ is a supporting hyperplane to $\Sigma$ at $x_0$. That is,
$\ell(x)\leq 0$ for all $x\in \Sigma$ and $\ell(x_0)=0$.
We then have a map $\nu\in S^{n-1}\mapsto x_0\in \Sigma$
\[
\Phi:S^{n-1}\to \Sigma,\qquad \Phi(\nu)=x_0.
\] 
This map assigns to each vector $\nu\in S^{n-1}$ a unique point $\Phi(\nu)\in \Sigma$ (uniqueness follows from the strict convexity of $\Sigma$) so that $\Pi_\nu$ is a supporting hyperplane to $\Sigma$ at $\Phi(\nu)$. 
In addition, also from the strict convexity of $\Sigma$, and consequently the uniqueness of maximizer of $\varphi$, we have $\Phi(-\nu)=-\Phi(\nu)$.
$\Phi$ is called the support map.

The dual norm of $N$, denoted by $N^*(x)=\|x\|^*$ is defined as follows.
Since $\R^n$ is finite dimensional, each linear functional $\ell:\R^n\to \R$ can be represented as $\ell(x)=y\cdot x$ for a unique $y\in \R^n$.
Given $y\in \R^n$, the dual norm of $N$ is then $N^*(y)=\sup_{x\in \Sigma}|x\cdot y|$, and writing $y=\lambda \,\nu$ with $\nu \in S^{n-1}$, we get that $N^*(y)= \lambda\,\p(\nu)$.
Hence the dual norm sphere of $\Sigma$ is $\Sigma^*=\{y\in \R^n:N^*(y)=1\}=\{\nu/\p(\nu):\nu\in S^{n-1}\}$.
We recall the following.
\begin{lemma}\cite[Lemma 2.3]{caffarelli-huang:nonisotropicreflector}\label{lm:lemma 2.3 CH}
For each $x\in \Sigma$ and $\nu^*=\nu/\p(\nu)$ with $\nu\in S^{n-1}$ we have 
\begin{enumerate}
\item[(a)] $|x\cdot \nu^*|\leq 1$; and
\item[(b)] $x\cdot \nu^*=1$ if and only if $x\in \Phi(\nu)$.
\end{enumerate}
\end{lemma}

Since $\Sigma$ is a convex surface, for each $x\in \Sigma$, there is a supporting hyperplane to $\Sigma$ at $x$ and let $\nu(x)$ be the outer unit normal to such a supporting hyperplane.  If $\Sigma$ is such that at each $x$ one can pick a supporting hyperplane with normal $\nu(x)$ in such a way that $\nu(x)$ is continuous for all $x\in \Sigma$, that is, {\it $\Sigma$ has a continuous normal field,} then from the proof of \cite[Theorem 6.2]{caffarelli-gutierrez-huang:antennaannals}, $\Sigma$ has a unique tangent plane at each point; that is, $\Sigma$ is differentiable. 

With the notation from \cite{caffarelli-huang:nonisotropicreflector}, the Minkowski functional of $\Sigma$ defined by $M_\Sigma(x)=\inf\{r>0: x\in r\,B\}$ satisfies $M_\Sigma(x)=N(x)$, where $B=\{x\in \R^n:N(x)\leq 1\}$.
It is proved in \cite[Lemma 2.4]{caffarelli-huang:nonisotropicreflector} that $\Sigma$ has a continuous normal field $\nu(x)$ if and only if $N\in C^1\(\R^n\setminus \{0\}\)$; and $p(x):=\nabla N(x)=\nu(x)/\p(\nu(x))$. 
%In particular, if $\Sigma$ is strictly convex then $N\in C^1\(\R^n\setminus \{0\}\)$.
We also recall \cite[Lemma 2.5]{caffarelli-huang:nonisotropicreflector} saying that if $\Sigma$ is $C^1$, then $\Sigma^*$ is strictly convex.
%
%Thus from, if $\Sigma$ is strictly convex, then $N\in C^1\(\R^n\setminus \{0\}\)$ and the gradient $p(x):=\nabla N(x)=\nu(x)/\p(\nu(x))$, where $\nu(x)$ is the outer unit normal to $\Sigma$ at $x$. 
%
Also, if $\Sigma$ is strictly convex, then $\Sigma^*\in C^1$, \cite[Lemma 2.7]{caffarelli-huang:nonisotropicreflector}. 
Therefore, if $\Sigma$  is strictly convex and $C^1$, then $p:\Sigma\to \Sigma^*$ is an homeomorphism and $p^*\circ p=Id$; $p^*=\nabla N^*$.

\section{A vector Snell's law for anisotropic media}\label{sec:vector snell law for anisotropic media}
\setcounter{equation}{0}

We have two homogenous\footnote{In homogenous media light travels in straight lines. This follows from \cite[Equation (3.97)]{kline-kay:electromagneticwaves} because the function $H$ there is independent of $x,y,z$.} and anisotropic media $I$ and $II$ so that the surfaces for the wave fronts are
%\marginpar{what about\\this\\wording?} 
given by a norm $N_1$ in $I$, and given by a norm $N_2$ in $II$.\footnote{A wave front is a surface in 3d space described by $\psi(x)=c\,t$ where $\psi$ is a function, $c$ is the speed of light in vacuum and $t$ is time. This means the points $x$ on the wave front that are travelling for a time $t$ are located on the surface $\psi(x)=c\,t$; see \cite[Chapter II, Sec. 2]{kline-kay:electromagneticwaves}.} 
We are assuming that the norms $N_i\in C^1\(\R^n\setminus \{0\}\)$ and the corresponding unit spheres $\Sigma_i$ are strictly convex, $i=1,2$.
Suppose $I$ and $II$ are separated by a plane having normal $\nu$ from medium $I$ to medium $II$ as in Figure \ref{fig:snell law}.

\begin{figure}[h]
\includegraphics[scale=.5, angle=0]{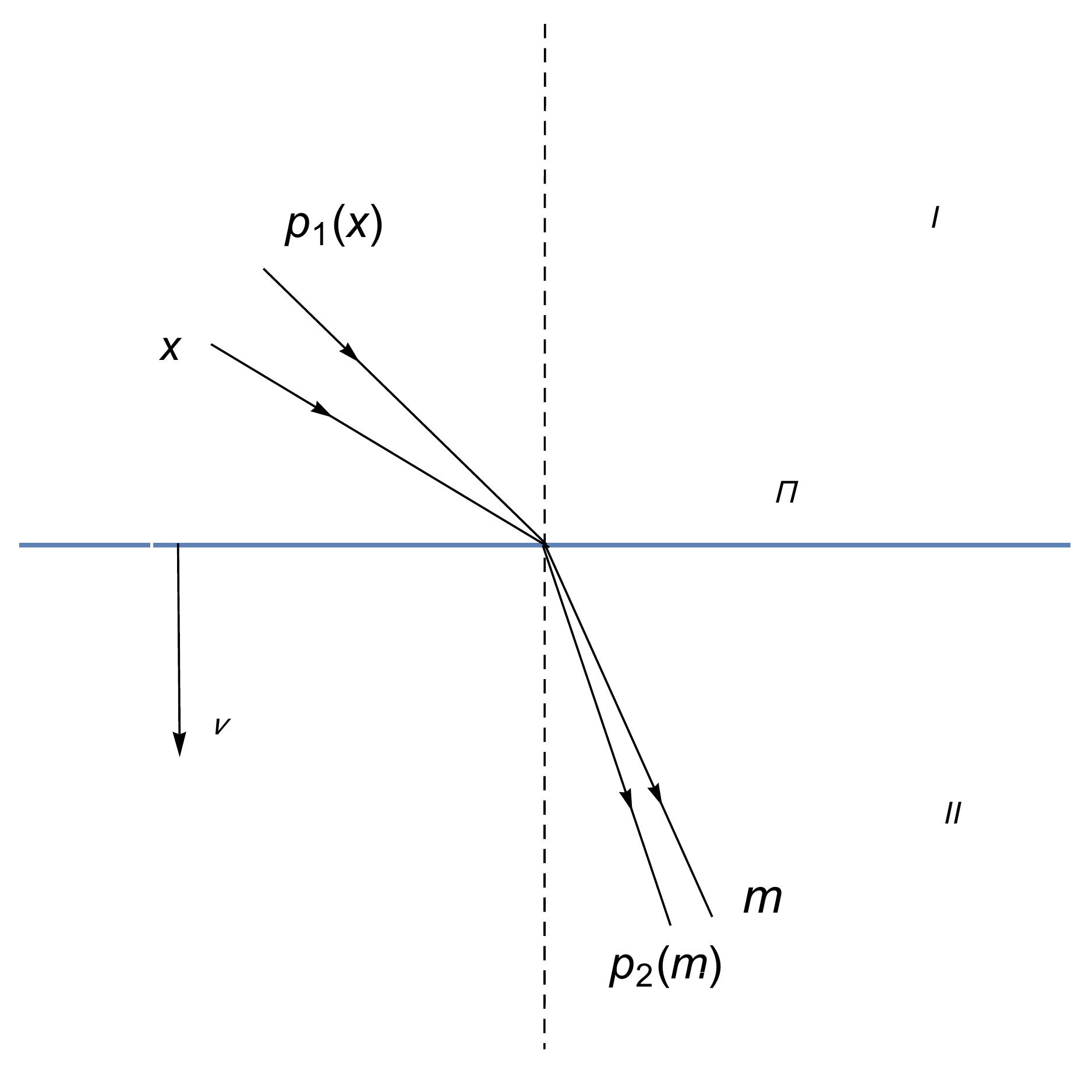}
\caption{Snell's law}
\label{fig:snell law}
\end{figure}

We formulate the Snell law in anisotropic media as follows: 
Each incident ray traveling in medium $I$ with direction $x\in \Sigma_1$
with $x\cdot \nu\geq 0$ and striking the plane $\Pi$ at some point $P_0$ is refracted in medium $II$ into a direction $m\in \Sigma_2$ if 
\begin{equation}\label{eq:vector snell law}
p_2(m)-p_1(x)\parallel \nu,
\end{equation}
with $p_i=\nabla N_i$, $i=1,2$; see Figure \ref{fig:snell law}. 
For each $x\in \Sigma_1$ with $x\cdot \nu\geq 0$, there is at most one $m\in \Sigma_2$ satisfying \eqref{eq:vector snell law} with $m\cdot \nu\geq 0$.
In fact, from \eqref{eq:vector snell law} there is $\lambda\in \R$ with $p_2(m)=p_1(x)+\lambda \,\nu$. That is, if $\ell$ is the line passing through $p_1(x)$ with direction $\nu$, then $p_2(m)\in \ell\cap \Sigma_2^*$. So if there were $m_1\neq m_2$ satisfying \eqref{eq:vector snell law}, then $p_2(m_i)\in \ell\cap \Sigma_2^*$, $i=1,2$.
The outer normal to $\Sigma_2^*$ at $p_2(m_i)$ equals to $\nabla N_2^*\(p_2(m_i)\)=p_2^*\(p_2(m_i)\)=m_i$ because $p_2:\Sigma_2\to \Sigma_2^*$ is a homeomorphism since $\Sigma_2$ is $C^1$ and strictly convex. And also $p_2(m_1)\neq p_2(m_2)$.
On the other hand, since the points $p_2(m_i)$ are also on the line $\ell$ and $\Sigma_2^*$ is strictly convex, the normals at $p_2(m_i)$ must satisfy that its dot products with $\nu$ must have different signs. Therefore, it cannot both happen $m_1\cdot \nu\geq 0$ and $m_2\cdot \nu\geq 0$. Therefore, only one of the two $m_i$ can satisfy $m_i\cdot \nu\geq 0$.

%\color{blue}
Physically, the norm $N(x)=1$ represents the location of the points $x$ after traveling for a time $t$, with $c\,t=1$, from the origin into the given medium. For example, if we are in an homogenous and isotropic medium with refractive index $n_1$, then the wave propagates from the origin with  velocity $v_1=c/n_1$. So if $x$ satisfies $N(x)=1$, then the Euclidean distance from $O$ to $x$ 
%that is, the optical path length from $O$ to $x$, 
must satisfy $|x|/t=v_1$. Since $t=1/c$, we obtain $|x|=v_1/c=1/n_1$ so $N(x)=n_1\,|x|$. Therefore, if medium $I$ has refractive index $n_1$ and medium $II$ has refractive index $n_2$, then $N_1(x)=n_1\,|x|$ and $N_2(x)=n_2\,|x|$.
We then have $p_i(x)=\nabla N_i(x)=n_i\,\dfrac{x}{|x|}$, $i=1,2$, $x\neq 0$, and so from \eqref{eq:vector snell law} we recover the standard Snell law: the unit incident direction $x$ is refracted into the unit direction $m$ when $n_1\,x-n_2\,m\parallel \nu$, see \cite[Formula (2.1)]{gutierrez-huang:farfieldrefractor}.

\color{black}
We shall prove that \eqref{eq:vector snell law} is equivalent to Fermat's principle of least time with respect to the norms $N_1,N_2$ where $\Sigma_1$ and $\Sigma_2$ strictly convex.
Suppose two anisotropic media $I,II$ with norms $N_1$ in $I$ and $N_2$ in $II$, are separated by a plane $\Pi$ as in Figure \ref{fig:snell law}. Given $X\in I$ and $Y\in II$, then Fermat's principle states that the (minimal) optical path from $X$ to $Y$ through the plane $\Pi$ 
is the path $XP_0Y$ where $P_0\in \Pi$ is the unique point such that
\begin{equation}\label{eq:Fermat principle}
\min \{N_1(P-X)+N_2(Y-P):P\in \Pi\}=N_1(P_0-X)+N_2(Y-P_0).
\end{equation}
Equation \eqref{eq:Fermat principle} implies that
\begin{equation}\label{eq:formulation with points}
p_2(Y-P_0)-p_1(P_0-X)\parallel \nu
\end{equation}
where $\nu$ is the normal to $\Pi$.
In fact, for $P\in \Pi$ we can write $P=P_0+\sum_{i=1}^{n-1}t_i\,e_i$ where $e_1,\cdots ,e_{n-1}$ is a basis for $\Pi$. From \eqref{eq:Fermat principle}
\begin{align*}
&\dfrac{\partial }{\partial t_j}\(N_1\(P_0-X+\sum_{i=1}^{n-1}t_i\,e_i\)+N_2\(P_0-Y+\sum_{i=1}^{n-1}t_i\,e_i\) \)\\
&= p_1\(P_0-X+\sum_{i=1}^{n-1}t_i\,e_i\)\cdot e_j+p_2\(P_0-Y+\sum_{i=1}^{n-1}t_i\,e_i\)\cdot e_j=0
\end{align*}
when $t_j=0$, for $j=1,\cdots ,n-1$. Since $p_i(-X)=-p_i(X)$, \eqref{eq:formulation with points} follows.
Also since $p_i$ are homogenous of degree zero, from \eqref{eq:formulation with points} we then obtain \eqref{eq:vector snell law} with 
$m=\dfrac{Y-P_0}{N_2(Y-P_0)}$ and $x=\dfrac{P_0-X}{N_1(P_0-X)}$.

Vice versa, from \eqref{eq:vector snell law} we deduce \eqref{eq:Fermat principle}.
In fact, let us fix $X\in I$ and $Y\in II$, and consider a path $XP_0Y$ with $P_0\in \Pi$.
The ray from $X$ to $P_0$ has direction $x=\dfrac{P_0-X}{N_1(P_0-x)}$, and the ray from $P_0$ to $Y$ has direction $m=\dfrac{Y-P_0}{N_2(Y-P_0)}$. 
If $x$ is refracted into $m$, then from \eqref{eq:vector snell law}, $p_2\(\dfrac{Y-P_0}{N_2(Y-P_0)}\)-p_1\(\dfrac{P_0-X}{N_1(P_0-x)}\)\parallel \nu$, and since $p_i$ are homogeneous of degree zero  
\begin{equation}\label{eq:snell without norm}
p_2\(Y-P_0\)-p_1\(P_0-X\)\parallel \nu. 
\end{equation}
Consider the functional $F(P)=N_1(P-X)+N_2(Y-P)$ for $P\in \Pi$. Write $P\in \Pi$ as $P=P_0+\sum_{i=1}^{n-1}t_i\,e_i$ where $e_1,\cdots ,e_{n-1}$ is a basis for $\Pi$.
We can write $F(P)=F(t_1,\cdots ,t_{n-1})$. As before and from \eqref{eq:snell without norm}
\[
\dfrac{\partial }{\partial t_j}F(t_1,\cdots ,t_{n-1})
= p_1\(P_0-X+\sum_{i=1}^{n-1}t_i\,e_i\)\cdot e_j+p_2\(P_0-Y+\sum_{i=1}^{n-1}t_i\,e_i\)\cdot e_j=0
\] 
for $j=1,\cdots ,n-1$ and $t_1=\cdots =t_{n-1}=0$.
On the other hand, since the functional $F$ on $\Pi$ is strictly convex, there are unique $t_1^0,\cdots ,t_{n-1}^0$ such that the minimum of $F$ is attained at $t_1^0,\cdots ,t_{n-1}^0$. Therefore, $t_1^0=\cdots =t_{n-1}^0=0$ and the point $P_0$ is the unique point in $\Pi$ minimizing $F$, obtaining Fermat's principle. 

\subsection{Physical constraints for refraction}\label{subsect:physical constraints} 
Since the incident ray $x$ is in media $I$ we must have $x\cdot \nu\geq 0$, where $\nu$ is the normal to the hyperplane separating $I$ and $II$, $\nu$ having direction from media $I$ to media $II$.
Similarly, since the refracted ray $m$ is in media $II$, we also have $m\cdot \nu\geq 0$; Figure \ref{fig:snell law}.

We analyze here the meaning of these two physical constraints $x\cdot \nu\geq 0$ and $m\cdot \nu\geq 0$ in the following two cases.

{\bf Case 1:}
Let us assume first that $\Sigma_1$ is contained inside $\Sigma_2$, i.e., $N_2(x)\leq 1$ for all $N_1(x)=1$. 
If $x\in \Sigma_1$, then $p_1(x)\in \Sigma_1^*$ and from Lemma \ref{lm:lemma 2.3 CH}(b) $x\cdot p_1(x)=1$.
Hence, if $x\in \Sigma_1$ and $m\in \Sigma_2$, then 
\begin{align*}
x\cdot \(p_1(x)-p_2(m)\)&=
1-x\cdot p_2(m)\\
&=
1-N_2(x)\,\dfrac{x}{N_2(x)}\cdot p_2(m)\geq 1-N_2(x)\qquad \text{from Lemma \ref{lm:lemma 2.3 CH}(a)}\\
&\geq 0\qquad \text{from the assumption $\Sigma_1$ contained inside $\Sigma_2$}.
\end{align*}
Thus, if $x\cdot \nu\geq 0$ and $x$ is refracted into $m$, then from \eqref{eq:vector snell law} $p_1(x)-p_2(m)=\lambda \,\nu$ and so $\lambda>0$.
Hence, $m\cdot \nu\geq 0$ and Lemma \ref{lm:lemma 2.3 CH}(b) imply that
\[
m\cdot p_1(x)-1=m\cdot \(p_1(x)-p_2(m)\)=\lambda\,m\cdot \nu\geq 0.
\]
Therefore, when $\Sigma_1$ is contained in the interior of $\Sigma_2$, we obtain the physical constraint
\begin{equation}\label{eq:physical constraint for m}
m\cdot p_1(x)\geq 1,
\end{equation}
for refraction of $x\in \Sigma_1$ into $m\in \Sigma_2$.

{\bf Case 2:}
Let us assume now that $\Sigma_2$ is contained inside $\Sigma_1$, i.e., $N_1(x)\leq 1$ for all $N_2(x)=1$. 
Reasoning as in the first case, we have
for $x\in \Sigma_1$ and $m\in \Sigma_2$, that 
\begin{align*}
m\cdot \(p_2(m)-p_1(x)\)&=
1-m\cdot p_1(x)\\
&=
1-N_1(m)\,\dfrac{m}{N_1(m)}\cdot p_1(x)\geq 1-N_1(m)\qquad \text{from Lemma \ref{lm:lemma 2.3 CH}(a)}\\
&\geq 0\qquad \text{from the assumption $\Sigma_2$ contained inside $\Sigma_1$}.
\end{align*}
Thus, if $m\cdot \nu\geq 0$ and $x$ is refracted into $m$, then from \eqref{eq:vector snell law} $p_2(m)-p_1(x)=\lambda \,\nu$ and so $\lambda>0$.
Hence, $x\cdot \nu\geq 0$ and Lemma \ref{lm:lemma 2.3 CH}(b) imply that
\[
x\cdot p_2(m)-1=x\cdot \(p_2(m)-p_1(x)\)=\lambda\,x\cdot \nu\geq 0.
\]
Therefore, when $\Sigma_2$ is contained in the interior of $\Sigma_1$, we obtain the physical constraint
\begin{equation}\label{eq:physical constraint for x}
x\cdot p_2(m)\geq 1,
\end{equation}
for refraction of $x\in \Sigma_1$ into $m\in \Sigma_2$.

%\color{blue}
Notice that, as explained before, if medium $I$ has refractive index $n_1$ and medium $II$ has refractive index $n_2$, then $N_1(x)=n_1\,|x|$ and $N_2(x)=n_2\,|x|$; and 
$p_i(x)=n_i\,\dfrac{x}{|x|}$, $i=1,2$.
Hence, if $n_1>n_2$ we are in Case 1 above, and so \eqref{eq:physical constraint for m}
reads $x\cdot m\geq n_2/n_1$ for $x,m$ unit vectors. If $n_1<n_2$ then we are then in Case 2, and so \eqref{eq:physical constraint for x}
reads $x\cdot m\geq n_1/n_2$ for $x,m$ unit vectors.
Therefore, when the media $I$ and $II$ are homogenous and isotropic we recover the physical constraints showed in \cite[Lemma 2.1]{gutierrez-huang:farfieldrefractor}.

\color{black}
%Given $\nu\in S^{n-1}$ there is a supporting hyperplane $T_\nu$ to $\Sigma$ at some point $x_0\in \Sigma$ having normal $\nu$ . The function $\p(\nu)=\sup_{x\in \Sigma}\nu\cdot x$ is the support function of $\Sigma$ and we have that $\p(\nu)=\nu\cdot x_0$.
%
%The 
%

\section{Uniformly refracting surfaces}\label{sec:Uniformly refracting surfaces}
\setcounter{equation}{0}
In this section, we shall describe the surfaces separating two anisotropic materials $I$ and $II$, like in Section \ref{sec:vector snell law for anisotropic media}, so that rays emanating from a point source, the origin, located in medium $I$ are refracted in medium $II$ into a fixed direction $m\in \Sigma_2$. 
These surfaces will have the form
\begin{equation}\label{eq:general refracting surfaces}
\left\{X\in \R^n: N_1(X)=p_2(m)\cdot X+b \right\},
\end{equation}
where $b\in \R$. If we write $X=\rho(x)\,x$ for $N_1(x)=1$, then the polar radius
\[
\rho(x)=\dfrac{b}{1-x\cdot p_2(m)}\qquad \text{for $x\in \Sigma_1$}.
\]
To show that these surfaces do the desired refraction job,  as in Section \ref{subsect:physical constraints} we distinguish two cases.

{\bf Case I:} $\Sigma_1$ is strictly contained in the interior of $\Sigma_2$, that is,
\begin{equation}\label{eq:kappa<1}
\kappa=\sup_{N_1(x)=1}N_2(x)<1.
\end{equation}
In this case, given $m\in \Sigma_2$, the desired surface is 
\begin{equation}\label{eq:uniformly refracting surface k<1}
S_I(m,b)=\left\{\rho(x)x:\rho(x)=\dfrac{b}{1-x\cdot p_2(m)},\text{ $x\in \Sigma_1$ with $m\cdot p_1(x)\geq 1$}\right\}
\end{equation}
with $b>0$. In fact, to verify that each ray with direction $x\in \Sigma_1$ such that $m\cdot p_1(x)\geq 1$ is refracted by $S_I(m,b)$ into $m$, we need to verify that \eqref{eq:vector snell law} holds, and the physical constraints $x\cdot \nu\geq 0$ and $m\cdot \nu\geq 0$ are met with $\nu$ the normal from medium $I$ to $II$. From \eqref{eq:general refracting surfaces}, the outward normal at a point $X$ is 
$\nu=p_1(x)-p_2(m)$ with $x=X/N_1(X)$ and so \eqref{eq:vector snell law} holds.
From Lemma \ref{lm:lemma 2.3 CH}
\begin{align}\label{eq:1-x p2m>1-kappa}
x\cdot \nu&=x\cdot \(p_1(x)-p_2(m)\)=1-x\cdot p_2(m)\notag\\
&=
1-N_2(x)\,\dfrac{x}{N_2(x)}\cdot p_2(m)\geq 1-N_2(x)>1-\kappa>0
\end{align}
from \eqref{eq:kappa<1}. Also $m\cdot \nu=m\cdot \(p_1(x)-p_2(m)\)=
m\cdot p_1(x)-m\cdot p_2(m)=m\cdot p_1(x)-1\geq 0$ by the definition of $S_I$.
%DO WE NEED $\kappa<1$? FOR THE PHYSICAL CONDITION WOULD BE ENOUGH $\kappa\leq 1$? OR IS THIS ASSUMPTION TO PREVENT $x\cdot p_2(m)=1$?

{\bf Case II:} $\Sigma_2$ is strictly contained in the interior of $\Sigma_1$, that is,
\begin{equation}\label{eq:kappa>1}
\kappa=\inf_{N_1(x)=1}N_2(x)>1.
\end{equation}
In this case, given $m\in \Sigma_2$, the desired surface is 
\begin{equation}\label{eq:uniformly refracting surface k>1}
S_{II}(m,b)=\left\{\rho(x)x:\rho(x)=\dfrac{b}{x\cdot p_2(m)-1},\text{ $x\in \Sigma_1$ with $x\cdot p_2(m)> 1$}\right\}
\end{equation}
with $b>0$. In fact and once again, to verify that each ray with direction $x\in \Sigma_1$ such that $x\cdot p_2(m)> 1$ is refracted by $S_{II}(m,b)$ into $m$, we need to verify that \eqref{eq:vector snell law} holds, and the physical constraints $x\cdot \nu\geq 0$ and $m\cdot \nu\geq 0$ are met with $\nu$ the normal towards medium $II$. From the definition of $S_{II}$, the outward normal at a point $X$ is 
$\nu=p_2(m)-p_1(x)$ with $x=X/N_1(X)$ and so \eqref{eq:vector snell law} holds.
From Lemma \ref{lm:lemma 2.3 CH}
\begin{align*}
m\cdot \nu&=m\cdot \(p_2(m)-p_1(x)\)=1-m\cdot p_1(x)\\
&=
1-N_1(m)\,\dfrac{m}{N_1(m)}\cdot p_1(x)\geq 1-N_1(m)\geq 1-\dfrac{1}{\kappa}\,N_2(m)=
1-\dfrac{1}{\kappa}>0
\end{align*}
from \eqref{eq:kappa>1}. Also $x\cdot \nu=x\cdot \(p_2(m)-p_1(x)\)=
x\cdot p_2(m)-1> 0$ by the definition of $S_{II}$.

%\color{blue}
\begin{remark}\rm
If medium $I$ is homogeneous and isotropic with refractive index $n_1$, then $N_1(x)=n_1\,|x|$. Also, if $II$ is also similar with refractive index $n_2$, then $N_2(x)=n_2\,|x|$. 
In this case, condition \eqref{eq:kappa<1} is equivalent to $n_1>n_2$, and the surface $S_{I}(m,b)$ is a half ellipsoid of revolution with axis $m$, recovering the surfaces from \cite[Formula (2.8)]{gutierrez-huang:farfieldrefractor}.
Similarly, condition \eqref{eq:kappa>1} is equivalent to $n_1<n_2$, and $S_{II}(m,b)$ is one of the branches of a hyperboloid of two sheets as in \cite[Formula (2.9)]{gutierrez-huang:farfieldrefractor}.
\end{remark}

\color{black}
\section{The refractor problem when $\kappa<1$, $\kappa$ in \eqref{eq:kappa<1}}\label{sec:kappa less than one}
\setcounter{equation}{0}

Using the uniformly refracting surfaces introduced in Section \ref{sec:Uniformly refracting surfaces}, we state and solve here the refraction problem we are interested in.

We are given two closed domains $\Omega_1\subset \Sigma_1$, $\Omega_2\subset \Sigma_2$, a non negative function $f\in L^1(\Omega_1)$, and a Radon measure $\mu$ in $\Omega_2$ satisfying the following conditions:
\begin{enumerate}
\item[(a)] the surface measure of the boundary of $\Omega_1$ is zero;
\item[(b)] $\inf_{x\in \Omega_1,m\in \Omega_2}m\cdot p_1(x)\geq 1$;
\item[(c)] $\Sigma_1$ and $\Sigma_2$ are $C^1$ and strictly convex.
\end{enumerate}
 
Refractors are then defined as follows.
\begin{definition}\label{def:refractor}
The surface $S=\{\rho(x)\,x:x\in \Omega_1\}$, with $\rho\in C(\Omega_1)$, $\rho>0$, is a refractor from $\Omega_1$ to $\Omega_2$, if for each $x_0\in \Omega_1$ there exist $m\in \Omega_2$ and $b>0$ such that the surface $S_I(m,b)$ supports $S$ at $x_0$, that is,
\[
\rho(x)\leq \dfrac{b}{1-x\cdot p_2(m)}\quad \text{for all $x\in \Omega_1$ with equality at $x=x_0$}.
\] 

The refractor mapping associated with the refractor $S$ is the set valued function 
\begin{equation}\label{def:refractor map}
\mathcal R_S(x_0)=\{m\in \Omega_2: \text{there exists $S_I(m,b)$ supporting $S$ at $x_0$}\}.
\end{equation}
\end{definition} 
We have the following lemma.
\begin{lemma}\label{lm:refractors are Lipschitz continuous}
If a refractor $S$ is parametrized by $\rho(x)\,x$, then $\rho$ is Lipschitz continuous in $\Omega_1$.
\end{lemma} 

\begin{proof}
Let $x_0,x\in \Omega_1$ and $S_I(m,b)$ supporting $S$ at $x_0$.
Then
\begin{align*}
\rho(x)-\rho(x_0)&\leq \dfrac{b}{1-x\cdot p_2(m)}- \dfrac{b}{1-x_0\cdot p_2(m)}=b\,\dfrac{(x-x_0)\cdot p_2(m)}{\(1-x\cdot p_2(m)\)\(1-x_0\cdot p_2(m)\)}\\
&= \rho(x_0)\,\dfrac{(x-x_0)\cdot p_2(m)}{1-x\cdot p_2(m)}
\leq
\max_{\Omega_1}\rho\,\dfrac{|x-x_0|\,|p_2(m)|}{1-\kappa}\qquad \text{from \eqref{eq:1-x p2m>1-kappa}}\\
&\leq C\,|x-x_0|
\end{align*}
since $N_i$ and $|\cdot |$ are all equivalent norms, $i=1,2$.
Reversing the roles of $x$ and $x_0$ we obtain the lemma.
\end{proof}

Following the notation from  
\cite[Section 2]{gutierrez-huang:nearfieldrefractor},
we denote by $C_S(\Omega_1,\Omega_2)$ the class of set-valued maps $\Phi:\Omega_1\to \Omega_2$ that are single valued for a.e. $x\in \Omega_1$, with respect to $f\,dx$, that are continuous in $\Omega_1$, and $\Phi(\Omega_1)=\Omega_2$.
Continuity of $\Phi$ at $x_0\in \Omega_1$ means that if $x_k\to x_0$ and $y_k\in \Phi(x_k)$, then there is a subsequence $y_{k_j}$ and $y_0\in \Phi(x_0)$ such that $ y_{k_j}\to y_0$.

\begin{lemma}\label{lm:refractors are C_S}
If $S$ is a refractor from $\Omega_1$ to $\Omega_2$, then the refractor map $\mathcal R_S\in C_S(\Omega_1,\Omega_2)$.
\end{lemma}
\begin{proof}
Let $m\in \Omega_2$ and define $b=\max_{x\in \Omega_1}\(\rho(x)\,(1-x\cdot p_2(m) \)$.
From \eqref{eq:1-x p2m>1-kappa}, $b\geq \max_{x\in \Omega_1}\(\rho(x)\,(1-\kappa )\)>0$.
Also, there is $x_0\in \Omega_1$ with $b=\rho(x_0)\,\(1-x_0\cdot p_2(m)\)$ and so $m\in \mathcal R_S(x_0)$, showing that  $\mathcal R_S(\Omega_1)=\Omega_2$.

Next, let us show that $\mathcal R_S(x)$ is single valued for a.e. $x\in \Omega_1$.
In fact, if at $x_0\in \Omega_1$ there exist $m_1\neq m_2\in \Omega_2$ with $m_i\in \mathcal R_S(x_0)$, $i=1,2$, then $x_0$ is a singular point to the surface $S$.
Otherwise, since $S_I(m_i,b_i)$, $i=1,2$ support $S$ at $x_0$, they would have the same tangent plane at $x_0\rho(x_0)$. Therefore, by the Snell law and since there is at most one $m$ satisfying \eqref{eq:vector snell law}, we obtain $m_1=m_2$.
From Lemma \ref{lm:refractors are Lipschitz continuous}, $S$ is Lipschitz, and since $|\partial \Omega_1|=0$, we obtain that $\mathcal R_S(x)$ is single valued a.e. in $\Omega_1$. 

It remains to show that $\mathcal R_S$ is continuous.
Let $x_i\to x_0\in \Omega_1$ and let $m_i\in \mathcal R_S(x_i)$.
Hence $\rho(x)\leq \dfrac{b_i}{1-x\cdot p_2(m_i)}$ for all $x\in \Omega_1$ with equality at $x=x_i$.
As before, $b_i=\rho(x_i)\,\(1-x_i\cdot p_2(m_i)\)\geq (1-\kappa)\,\min_{\Omega_1}\rho $, and $b_i=\rho(x_i)\(1-N_2(x_i)\,\dfrac{x_i}{N_2(x_i)}\cdot p_2(m_i)\)\leq \max_{\Omega_1}\rho\,(1+\kappa)$ from Lemma \ref{lm:lemma 2.3 CH}(a) and \eqref{eq:kappa<1}.
We have $m_i\in \Omega_2\subset \Sigma_2$ and $p_2\in C\(\Sigma_2\)$.
By compactness there are subsequences $m_{i_k}\to m_0\in \Omega_2$ and $b_{i_k}\to b_0>0$ so that $\rho(x)\leq \dfrac{b_0}{1-x\cdot p_2(m_0)}$ for all $x\in \Omega_1$ with equality at $x=x_0$.
This completes the proof to the lemma.

\end{proof}
Using \cite[Lemma 2.1]{gutierrez-huang:nearfieldrefractor}, we obtain from Lemma \ref{lm:refractors are C_S} that if $S$ is a refractor from $\Omega_1$ to $\Omega_2$, then the set function
\begin{equation}\label{def:definition of refractor measure}
\mathcal M_{S,f}(E)=\int_{\mathcal R_S^{-1}(E)}f(x)\,dx
\end{equation}
is a Borel measure in $\Omega_2$, that {\it is called the refractor measure}.

Continuing using the set up from \cite[Section 2]{gutierrez-huang:nearfieldrefractor}, we recall \cite[Definition 2.2]{gutierrez-huang:nearfieldrefractor}: given $\mathcal F\subset C(\Omega_1)$ and a map $\mathcal T:\mathcal F\to C_S(\Omega_1,\Omega_2)$, we say $\mathcal T$ is continuous at $\phi\in \mathcal F$ if whenever $\phi_j\in \mathcal F$, $\phi_j\to \phi$ uniformly in $\Omega_1$, $x_0\in \Omega_1$ and $y_j\in \mathcal T(\phi_j)(x_0)$, then there exists a subsequence $y_{j_\ell}\to y_0$ with $y_0\in \mathcal T(\phi)(x_0)$.    
If we let 
\begin{equation}\label{eq:definition of mathcal F}
\mathcal F=\{\rho\in C(\Omega_1): S_\rho \text{ is a refractor from $\Omega_1$ to $\Omega_2$ parametrized by $\rho(x)x$}\}
\end{equation}
then we have the following lemma.
\begin{lemma}\label{lm:continuity of T}
The mapping $\mathcal T:\mathcal F\to C_S(\Omega_1,\Omega_2)$ defined by $\mathcal T(\rho)=\mathcal R_{S_\rho}$ is continuous at each $\rho\in \mathcal F$. 
\end{lemma}
\begin{proof}
Let $\rho_j, \rho\in \mathcal F$ with $\rho_j\to \rho$ uniformly in $\Omega_1$, $x_0\in \Omega_1$ and $m_j\in \mathcal R_{S_{\rho_j}}(x_0)$. Hence $\rho_j(x)\leq \dfrac{b_j}{1-x\cdot p_2(m_j)}$ for all $x\in \Omega_1$ with equality at $x=x_0$.
As in the last part of the proof of Lemma \ref{lm:refractors are C_S}, $b_j$ are bounded away from $0$ and $\infty$.
Therefore there exist subsequences $b_{j_\ell}\to b$ and $m_{j_\ell}\to m$ with 
$\rho(x)\leq \dfrac{b}{1-x\cdot p_2(m)}$ for all $x\in \Omega_1$ with equality at $x=x_0$. Thus $m\in \mathcal R_{S_\rho}(x_0)$ and we are done. 
\end{proof}
As a consequence of Lemma \ref{lm:continuity of T} we obtain from \cite[Lemma 2.3]{gutierrez-huang:nearfieldrefractor} that 
\[
\text{if $\rho_j\to \rho$ uniformly in $\Omega_1$, then $\mathcal M_{S_{\rho_j},f}\to \mathcal M_{S_{\rho},f}$ weakly.}
\]
In addition, properties (A1)-(A3) from \cite[Section 2.1]{gutierrez-huang:nearfieldrefractor} translate to the present case as follows:
\begin{enumerate}
\item[(A1)] if $S_{\rho_1}$ and $S_{\rho_2}$ are refractors from $\Omega_1$ to $\Omega_2$, 
then $S_{\rho_1\wedge \rho_2}$ is a refractor from $\Omega_1$ to $\Omega_2$ with $\rho_1\wedge \rho_2=\min\{\rho_1,\rho_2\}$;
\item[(A2)] if $\rho_1(x_0)\leq \rho_2(x_0)$, then $\mathcal R_{S_{\rho_1}}(x_0)\subset \mathcal R_{S_{\rho_1\wedge \rho_2}}(x_0)$;
\item[(A3)] we let $h_{b,m}(x)=\dfrac{b}{1-x\cdot p_2(m)}$, and we have 
%\marginpar{correct\\notation?}
\[
\{h_{b,m}(x):m\in \Omega_2,0<b<\infty\}\subset \mathcal F,
\]
with $\mathcal F$ defined by \eqref{eq:definition of mathcal F}.
In addition we have the following
\begin{enumerate}
\item[(a)] $m\in \mathcal R_{S_{h_{b,m}}}(x)$ for all $x\in \Omega_1$, from Section \ref{sec:Uniformly refracting surfaces}, Case I;
\item[(b)] $h_{b_1,m}\leq h_{b_2,m}$ for $b_1\leq b_2$;
\item[(c)] $h_{b,m}\to 0$ uniformly in $\Omega_1$ as $b\to 0$;
\item[(d)] $h_{b,m}\to h_{b_0,m}$ uniformly in $\Omega_1$ as $b\to b_0$.
\end{enumerate}
\end{enumerate}

We then introduce the following definition.
\begin{definition}
Let $f\in L^1(\Omega_1)$ and let $\mu$ be a Radon measure in $\Omega_2$ with $\int_{\Omega_1}f\,dx=\mu(\Omega_2)$.
The refractor $S$ from $\Omega_1$ to $\Omega_2$ is a weak solution of the refractor problem if 
\[
\mathcal M_{S,f}(E)=\mu(E)
\]
for each Borel set $E\subset \Omega_2$, where $\mathcal M_{S,f}$ is the refractor measure defined by 
\eqref{def:definition of refractor measure}. 
\end{definition}

Using the above set up and the existence results from \cite[Section 2]{gutierrez-huang:nearfieldrefractor}
we obtain the following theorems showing solvability of the refractor problem for anisotropic media when $\kappa<1$.
We first show solvability when the measure $\mu$ is discrete.

\begin{theorem}\label{thm:discrete case}
Let $f\in L^1(\Omega_1)$ with $f>0$ a.e., $m_1,\cdots ,m_N\in \Omega_2$ be distinct points, and $g_1,\cdots ,g_N$ positive numbers satisfying $\int_{\Omega_1}f\,dx=\sum_{i=1}^Ng_i$.

Then for each $0<b_1<\infty$ there exist unique positive $b_2,\cdots ,b_N$ such that 
\[
S=\{\rho(x)x:x\in \Omega_1\}\text{ with } \rho(x)=\min_{1\leq i\leq N}h_{b_i,m_i}(x),
\]
is a weak solution to the refractor problem. In addition, $\mathcal M_{S,f}(\{m_i\})=g_i$ for $1\leq i\leq N$.
\end{theorem}
\begin{proof}
To prove this theorem, we use \cite[Theorem 2.5]{gutierrez-huang:nearfieldrefractor} with the set up from above, for which we need to verify that the assumptions of that theorem are met.
In fact, we need to show that we can choose positive numbers $b_2^0,\cdots ,b_N^0$ such that $\rho_0(x)=\min_{1\leq i\leq N}h_{b_i^0,m_i}(x)$ such that 
$\mathcal M_{S_{\rho_0},f}(m_i)\leq g_i$ for $2\leq i\leq N$, where $b_1^0=b_1$.
We have $h_{b_1^0,m_1}(x)=\dfrac{b_1^0}{1-x\cdot p_2(m_1)}\leq \dfrac{b_1}{1-\kappa}$ from \eqref{eq:1-x p2m>1-kappa}.
Also $h_{b_i^0,m_i}(x)=\dfrac{b_i^0}{1-x\cdot p_2(m_i)}\geq \dfrac{b_i^0}{1+\kappa}$ for $2\leq i\leq N$
from Lemma \ref{lm:lemma 2.3 CH}(b) and \eqref{eq:kappa<1}.
Therefore choosing $b_2^0,\cdots ,b_N^0$ suitable so that $\rho_0(x)=h_{b_1^0,m_1}(x)$, the assumptions of
 \cite[Theorem 2.5]{gutierrez-huang:nearfieldrefractor} are met and the existence follows.
 The uniqueness follows from \cite[Theorem 2.7]{gutierrez-huang:nearfieldrefractor} since $f>0$ a.e.
\end{proof}

We are now ready to prove the following existence theorem for a general Radon measure $\mu$.

\begin{theorem}\label{thm:existence general measure}
Let $f\in L^1(\Omega_1)$ with $f>0$ a.e, and let $\mu$ be a Radon measure in $\Omega_2$ such that $\int_{\Omega_1}f(x)\,dx=\mu\(\Omega_2\)$.
Then for each $x_0\in \Omega_1$ and $R_0>0$, there exists $\mathcal S$ weak solution to the refractor problem passing through the point $X_0=R_0\,x_0$.
\end{theorem}
\begin{proof}
Let $\mu_\ell=\sum_{i=1}^Ng_i\,\delta_{m_i}$ be a sequence of discrete measures with $\mu_\ell\to \mu$ weakly and $\mu_\ell\(\Omega_2\)=\mu\(\Omega_2\)$ for $\ell=1,2,\cdots $.
From Theorem \ref{thm:discrete case} and for the measure $\mu_\ell$, there exists a refractor $S_{\rho_\ell^*}$ parametrized by $\rho_\ell^*$.
Notice that $S_{C_\ell\, \rho_\ell^*}$ is also a solution to the same refractor problem since $\mathcal R_{C_\ell\, \rho_\ell^*}
=\mathcal R_{\rho_\ell^*}$ for each positive constant $C_\ell$.
Then pick $C_\ell$ so that $C_\ell\,\rho^*_\ell(x_0)=R_0$.
Now we use the existence result \cite[Theorem 2.8]{gutierrez-huang:nearfieldrefractor}, and in order to do that we need to verify that the hypotheses (i) and (ii) of that theorem hold in the present case.
To verify (i) we show that if $R_1\in \text{Range}\(h_{b,m}\)$, then 
\[
R_1\,\dfrac{1-\kappa}{1+\kappa}\leq h_{b,m}\leq R_1\,\dfrac{1+\kappa}{1-\kappa}.
\]
In fact, there exists $x_1\in \Omega_1$ with $R_1=h_{b,m}(x_1)=\dfrac{b}{1-x_1\cdot p_2(m)}$,
so $h_{b,m}(x)=\dfrac{b}{1-x\cdot p_2(m)}=R_1\,\dfrac{1-x_1\cdot p_2(m)}{1-x\cdot p_2(m)}$,
and the desired inequalities follow from Lemma \ref{lm:lemma 2.3 CH}(b) and \eqref{eq:kappa<1}.
The verification of (ii), that is, the family $\{\rho\in \mathcal F:C_0\leq \rho\leq C_1\}$ is compact in $C(\Omega_1)$, follows from Lemma \ref{lm:refractors are Lipschitz continuous} and the proof of Lemma \ref{lm:continuity of T}.
\end{proof}

\begin{remark}\label{rmk:kappa bigger than one}\rm
In the same way we can state and solve the refractor problem when $\kappa>1$, i.e., \eqref{eq:kappa>1} holds, using instead the uniformly refracting surfaces $S_{II}(m,b)$ defined by \eqref{eq:uniformly refracting surface k>1}.
Now the functions $h_{b,m}$ are defined by $h_{b,m}(x)=\dfrac{b}{x\cdot p_2(m)-1}$ and the properties (A1)-(A3) 
defined after Lemma \ref{lm:continuity of T} must be changed in accordance with properties (A1')-(A3') in \cite[Section 2.2]{gutierrez-huang:nearfieldrefractor}. All lemmas in this section then hold true with obvious changes.
For the existence of solutions we now need to use \cite[Theorems 2.9 and 2.11]{gutierrez-huang:nearfieldrefractor}.

In Theorem \ref{thm:existence general measure} uniqueness follows from optimal mass transport, Section \ref{sec:optimal mass transport}.
\end{remark}

\section{Propagation of light in anisotropic materials}\label{sec:propagation light in anistropic materials}
\setcounter{equation}{0}
We begin this section with some background on the propagation of light in anisotropic materials.
Let us assume we have a material whose permittivity and permeability are given by positive definite and symmetric matrices $\epsilon(x,y,z)$ and $\mu(x,y,z)$, respectively.
Assuming we are in the geometric optics regime, i.e., the wave length of the radiation is very small compared with the objects considered, it is known \cite[Chap. III, Sect. 4]{kline-kay:electromagneticwaves} that the function $\psi=\psi(x,y,z)$ defining the wave fronts $\psi(x,y,z)$=constant, satisfies the following first order pde, the Fresnel differential equation:
\begin{equation}\label{eq:Fresnel differential equation of the wave fronts}
\det
\left(
\begin{matrix}
\epsilon & R\\
-R & \mu
\end{matrix}
 \right)=0,
\end{equation}
where $R$ is the $3\times 3$ skew-symmetric matrix
\[
R
=
\left( 
\begin{matrix}
0 & -\psi_z & \psi_y\\
\psi_z & 0 & -\psi_x\\
-\psi_y & \psi_x & 0
\end{matrix}
\right).
\]
%Since by definition the light rays are the orthogonal trajectories to the wave fronts, the normals to the wave fronts, or wave normals, give the direction of the light rays.
%In anisotropic media, light rays do not follow in the direction of the wave normals. 

We can re write Fresnel's equation in a simpler form using the following Schur's determinant identity:
{\it if $A$ is an $n\times n$ invertible matrix, $B$ is $n\times m$, $C$ is $m\times n$ and $D$ is $m\times m$, then
\[
\det \left(
\begin{matrix}
A & B\\
C & D
\end{matrix}
\right) 
=
\det A\,\det \left(D-C\,A^{-1}\,B \right)
=
\det D\,\det \left(A-B\,D^{-1}\,C \right),
\]
}
of course for the last identity $D$ is invertible.
We then get
\[
\det\left(
\begin{matrix}
\epsilon & R\\
-R & \mu
\end{matrix}
\right)=\det \epsilon \,\det \left( \mu+ R\,\epsilon^{-1}\,R\right) 
=
\det \mu \,\det \left( \epsilon+ R\,\mu^{-1}\,R\right)
\]
and since $\epsilon,\mu$ are positive definite, \eqref{eq:Fresnel differential equation of the wave fronts} is equivalent to either
\begin{equation}\label{eq:Fresnel equation of wave fronts simpler}
\det \left( \mu+ R\,\epsilon^{-1}\,R\right)=0, 
\end{equation}
or
\begin{equation}\label{eq:Fresnel equation of wave fronts simpler bis}
\det \left( \epsilon+ R\,\mu^{-1}\,R\right)=0. 
\end{equation}
Letting
\begin{equation}\label{eq:definition of tau}
\tau=\mu^{-1/2}\epsilon \mu^{-1/2},
\end{equation}
$\tau$ is symmetric and positive definite, so there is an orthogonal matrix $O$ and a diagonal matrix $D$ such that
\[
\tau =ODO^t.
\]
For a column vector $v=\left(\begin{matrix}v_1\\v_2\\v_3\end{matrix}\right)$ define
\[
\text{Skew}(v)=
\left( 
\begin{matrix}
0 & -v_3 & v_2\\
v_3 & 0 & -v_1\\
-v_2 & v_1 & 0
\end{matrix}
\right).
\]
Given a $3\times 3$ matrix $B$ we have the formula 
\begin{align*}
B^t\,\text{Skew}(B\,v)\,B&=\det B\, \text{Skew}(v).
%\\
%\left(\text{Skew}(v)\right)^2&=v\,v^t-v^t\,v\,Id=v\otimes v-(v\cdot v)\,Id.
\end{align*}
We then re write \eqref{eq:Fresnel equation of wave fronts simpler bis} as follows:
\begin{align*}
\det \left( \epsilon+ R\,\mu^{-1}\,R\right)
&=
\det \mu \,\det \left(\mu^{-1/2}\epsilon \mu^{-1/2}+\left(\mu^{-1/2}R\mu^{-1/2}\right)  \left(\mu^{-1/2}R\mu^{-1/2}\right)  \right).
\end{align*}
%and let 
%\[
%\tau=\mu^{-1/2}\epsilon \mu^{-1/2}.
%\]
%$\tau$ is symmetric and positive definite so there is an orthogonal matrix $O$ and a diagonal matrix $D$ such that
%\[
%\tau =ODO^t.
%\]
Also since $\mu^{-1/2}$ is symmetric, we have
\begin{align*}
\mu^{-1/2}R\mu^{-1/2}= \mu^{-1/2}\text{Skew}(\nabla \psi)\mu^{-1/2}
&=
\mu^{-1/2}\text{Skew}\left(\mu^{-1/2}\mu^{1/2}\nabla \psi\right)\mu^{-1/2}\\
&=\det \left(\mu^{-1/2} \right)\,\text{Skew}\left(\mu^{1/2}\nabla \psi\right).
\end{align*}
Also
\begin{align*}
&\mu^{-1/2}\epsilon \mu^{-1/2}+
\left(\mu^{-1/2}R\mu^{-1/2}\right)  \left(\mu^{-1/2}R\mu^{-1/2}\right)\\
&=
ODO^t+OO^t\left(\mu^{-1/2}R\mu^{-1/2}\right)OO^t  \left(\mu^{-1/2}R\mu^{-1/2}\right)OO^t\\
&=
O\left(D +O^t\left(\mu^{-1/2}R\mu^{-1/2}\right)OO^t  \left(\mu^{-1/2}R\mu^{-1/2}\right)O\right) O^t.
\end{align*}
We have
\begin{align*}
\bar R:=O^t\left(\mu^{-1/2}R\mu^{-1/2}\right)O
&=
\det \left(\mu^{-1/2} \right)
O^t\,\text{Skew}\left(\mu^{1/2}\nabla \psi\right)\,O\\
&=
\det \left(\mu^{-1/2} \right)
O^t\,\text{Skew}\left(O\,O^t\mu^{1/2}\nabla \psi\right)\,O\\
&=
\det \left(\mu^{-1/2} \right)\,\det O\,
\text{Skew}\left(O^t\mu^{1/2}\nabla \psi\right),
\end{align*}
so
\begin{align*}
D+\bar R \bar R
&=
D+\left(\det \left(\mu^{-1/2} \right) \right)^2\,
\left(\det O\right)^2
\,
\text{Skew}\left(O^t\mu^{1/2}\nabla \psi\right)\,
\text{Skew}\left(O^t\mu^{1/2}\nabla \psi\right)\\
&=
D+\left(\det \left(\mu^{-1/2} \right) \right)^2
\,
\text{Skew}\left(O^t\mu^{1/2}\nabla \psi\right)\,
\text{Skew}\left(O^t\mu^{1/2}\nabla \psi\right)\\
&=
D+
\text{Skew}\left(O^t \dfrac{\mu^{1/2}}{\det \left(\mu^{1/2} \right)} \nabla \psi\right)\,
\text{Skew}\left(O^t\dfrac{\mu^{1/2}}{\det \left(\mu^{1/2} \right)} \nabla \psi\right).
\end{align*}
Therefore 
\begin{align}\label{eq:formula for D plus skew p}
\det \left( \epsilon+ R\,\mu^{-1}\,R\right)
&=
\det \mu \,
\det \left(D+
\text{Skew}\left(O^t \dfrac{\mu^{1/2}}{\det \left(\mu^{1/2} \right)} \nabla S\right)\,
\text{Skew}\left(O^t\dfrac{\mu^{1/2}}{\det \left(\mu^{1/2} \right)} \nabla S\right) \right)\notag\\
&=
\det \mu \,
\det \left(D+\text{Skew}(p)\,\text{Skew}(p)\right)=0,
\end{align}
where
\[
p=\left(\begin{matrix}p_1\\p_2\\p_3 \end{matrix} \right):=O^t \dfrac{\mu^{1/2}}{\det \left(\mu^{1/2} \right)} \nabla S.
\]

%If for a column vector $p=\left(\begin{matrix}p_1\\p_2\\p_3\end{matrix}\right)$ we define
%\[
%\text{Skew}(p)=
%\left( 
%\begin{matrix}
%0 & -p_3 & p_2\\
%p_3 & 0 & -p_1\\
%-p_2 & p_1 & 0
%\end{matrix}
%\right),
%\]
%one can see using properties of skew-symmetric matrices that \eqref{eq:Fresnel equation of wave fronts simpler bis} 
%can be written as
%\begin{align}\label{eq:formula for D plus skew p}
%\det \left( \epsilon+ R\,\mu^{-1}\,R\right)
%&=
%\det \mu^{-1} \,
%\det \left(D+
%\text{Skew}\left(O^t \dfrac{\mu^{1/2}}{\det \left(\mu^{1/2} \right)} \nabla \psi\right)\,
%\text{Skew}\left(O^t\dfrac{\mu^{1/2}}{\det \left(\mu^{1/2} \right)} \nabla \psi\right) \right)\notag\\
%&=
%\det \mu^{-1} \,
%\det \left(D+\text{Skew}(p)\,\text{Skew}(p)\right)=0,
%\end{align}
%where
%\[
%p=\left(\begin{matrix}p_1\\p_2\\p_3 \end{matrix} \right):=O^t \dfrac{\mu^{1/2}}{\det \left(\mu^{1/2} \right)} \nabla \psi.
%\]
Notice that this calculation is done at a fixed point $(x,y,z)$ since the matrices $\epsilon$ and $\mu$ depend on the point $(x,y,z)$; therefore the matrices $D$ and $O$ depend also on $(x,y,z)$.
Next we have
\[
\text{Skew}(p)\,\text{Skew}(p)=p\,p^t-p^t\,p\,Id=p\otimes p -(p\cdot p)\,Id,
\]
so by \eqref{eq:formula for D plus skew p} the Fresnel equation for the wave fronts \eqref{eq:Fresnel equation of wave fronts simpler bis} is then
\[
0
=\det \left(D+\text{Skew}(p)\,\text{Skew}(p)\right)
=\det \left(D+p\otimes p -(p\cdot p)\,Id\right).
\]
To write this equation in a more convenient form,
set 
\[
D
=
\left(
\begin{matrix}
\tau_1 & 0 & 0\\
0 & \tau_2 & 0\\
0 & 0 &\tau_3
\end{matrix} 
\right),
\]
(a matrix depending on $(x,y,z)$),  
so
\begin{align*}
0&=\det \left(D+p\otimes p -(p\cdot p)\,Id\right)\\
&=
\det \left(
\begin{matrix}
\tau_1+p_1^2-|p|^2 & p_1p_2 & p_1p_3\\
p_2p_1 & \tau_2+p_2^2-|p|^2 & p_2p_3\\
p_3p_1 & p_3p_2 & \tau_3+p_3^2-|p|^2\\
\end{matrix} \right)\\
&=
\det \left(
\begin{matrix}
\tau_1-p_2^2-p_3^2 & p_1p_2 & p_1p_3\\
p_2p_1 & \tau_2-p_1^2-p_3^2 & p_2p_3\\
p_3p_1 & p_3p_2 & \tau_3-p_1^2-p_2^2\\
\end{matrix} \right).
\end{align*}
Let us now define for an arbitrary vector $(p_1,p_2,p_3)$ the following functions, {\it which depend on the point $(x,y,z)$ since $\tau_i$ depend on $(x,y,z)$}
\[
\Phi(p_1,p_2,p_3)= \dfrac12 \left(\dfrac{1}{\tau_2} + \dfrac{1}{\tau_3}\right)p_1^2 + \dfrac12\left(\dfrac{1}{\tau_1} + \dfrac{1}{\tau_3}\right)p_2^2+ \dfrac12\left(\dfrac{1}{\tau_1} + \dfrac{1}{\tau_2}\right)p_3^2,
\]
and
\[
\Psi(p_1,p_2,p_3)=(p_1^2+p_2^2+p_3^2)\left( \dfrac{1}{\tau_2\tau_3} p_1^2+\dfrac{1}{\tau_1\tau_3} p_2^2+\dfrac{1}{\tau_1\tau_2} p_3^2 \right).
\]
It is easy to check that
\begin{align}\label{eq:formula for determinant in terms of Phi and Psi}
&\dfrac{1}{\tau_1\tau_2\tau_3}\det\left[
 \begin{matrix}
 \tau_1-p_2^2-p_3^2 & p_1p_2 & p_1p_3\\
p_2p_1 & \tau_2-p_1^2-p_3^2 & p_2p_3\\
p_3p_1 & p_3p_2 & \tau_3-p_1^2-p_2^2\\
   \end{matrix}\right]
   =1 - 2\Phi(p_1,p_2,p_3)+\Psi(p_1,p_2,p_3).
   \end{align}
Now write 
\[
1-2\Phi+\Psi=1-2\Phi +\Phi^2-\Phi^2 +\Psi
=
(1-\Phi)^2-(\Phi^2-\Psi).
\]
Next notice that 
%by comparison (see proof below)\marginpar{no proof}
\begin{equation}\label{eq:Phi squared bigger than Psi}
\Phi^2\geq \Psi,
\end{equation}
which follows using Legrange multipliers since $\Phi^2- \Psi$ is homogenous of degree four. 
So we can write
\[
1-2\Phi+\Psi=\left(1-\Phi-\sqrt{\Phi^2-\Psi} \right)\,\left(1-\Phi+\sqrt{\Phi^2-\Psi}\right).
\]
We then obtain that the Fresnel equation of wave fronts \eqref{eq:Fresnel equation of wave fronts simpler bis} can be split as the following two equations
\begin{equation}\label{eq:fresnel equations split}
1-\Phi-\sqrt{\Phi^2-\Psi}=0\qquad \text{or} \qquad 1-\Phi+\sqrt{\Phi^2-\Psi}=0.
\end{equation}
Each of these equations describes a three dimensional surface that depends of the point $(x,y,z)$ chosen at the beginning; see Figure \ref{fig:fresnel surfaces}. That is, in this way each point $(x,y,z)$ in the space has associated a pair of surfaces, one enclosing the other. The inner surface is convex and the outer surface is neither convex nor concave.
We have shown that the vector 
\begin{equation}\label{eq:vector p rotation of gradient S}
p=\left(\begin{matrix}p_1\\p_2\\p_3 \end{matrix} \right):=O^t \dfrac{\mu^{1/2}}{\det \left(\mu^{1/2} \right)} \nabla \psi,
\end{equation}
belongs to one of the surfaces, 
with all quantities calculated at $(x,y,z)$, and the matrix $O$ is orthogonal and diagonalizes the matrix $\tau$. In other words, we have shown that the gradient $\nabla \psi(x,y,z)$ of the wave front $\psi$=constant, when multiplied by the matrix $\dfrac{\mu^{1/2}}{\det \left(\mu^{1/2} \right)}$ and conveniently rotated by $O^t$, belongs to one of the surfaces described by the equations \eqref{eq:fresnel equations split}.

\begin{figure}[h]
\includegraphics[scale=.5, angle=0]{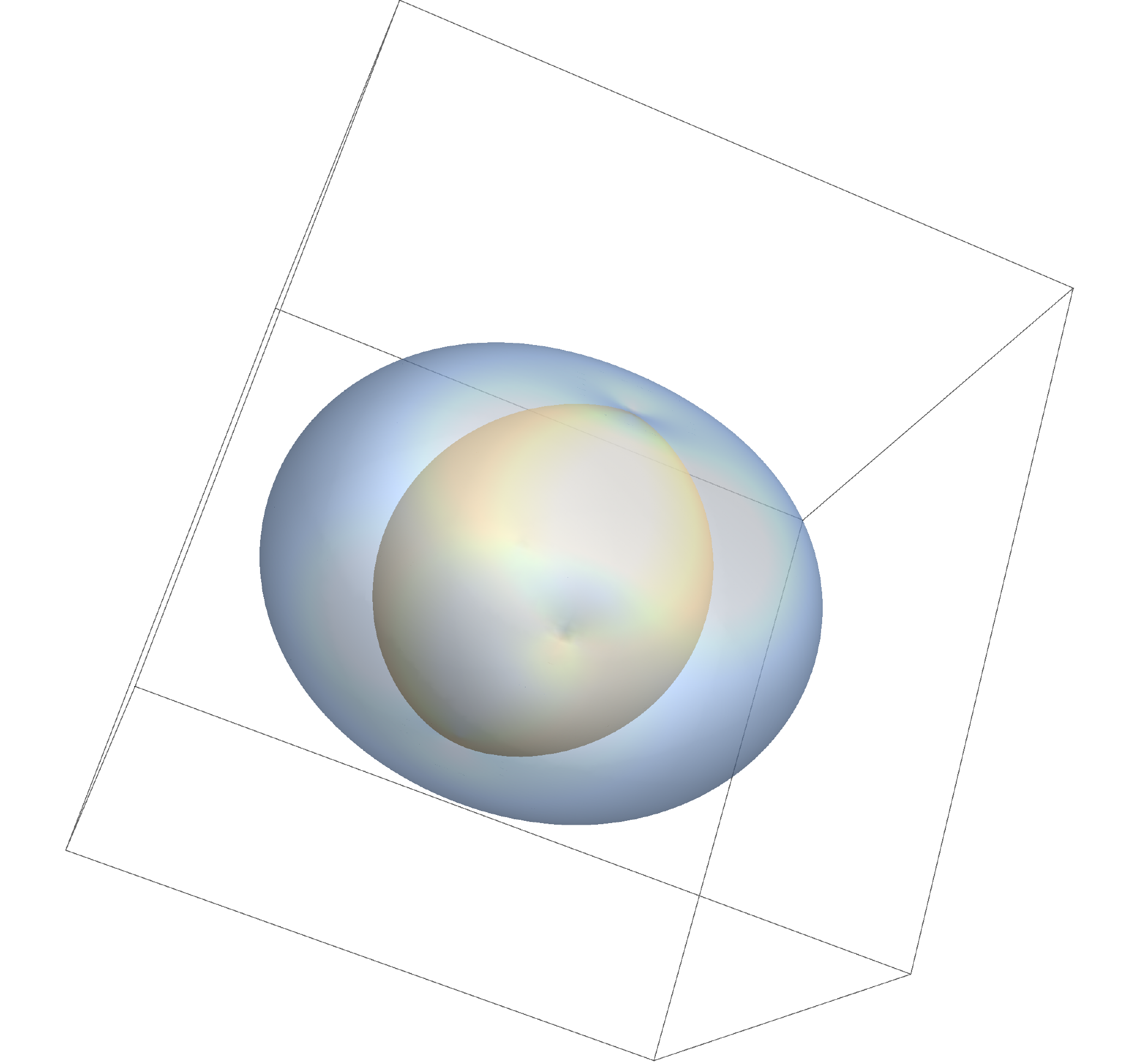}
\caption{Fresnel surfaces when $\tau_1=1,\tau_2=2,\tau_3=3$; $\mu=Id$}
\label{fig:fresnel surfaces}
\end{figure}

Notice that when the permittivity matrix is $\epsilon\,Id$ and the permeability matrix is $\mu\,Id$, where $\epsilon$ and $\mu$ are scalar functions depending only on position, then we recover the eikonal equation $|\nabla \psi|^2=\epsilon \,\mu$.
In fact, in this case the matrix $\tau=(\epsilon/\mu)\,Id$, so $\tau_i=\epsilon/\mu$ for $i=1,2,3$, $O=Id$, 
\[
\Phi(p_1,p_2,p_3)=\dfrac{\mu}{\epsilon}(p_1^2+p_2^2+p_3^2),
\]
and
\[
\Psi(p_1,p_2,p_3)=\left( \dfrac{\mu}{\epsilon}\right)^2(p_1^2+p_2^2+p_3^2)^2.
\]
So $\Phi^2=\Psi$ and both surfaces in \eqref{eq:fresnel equations split}
are identically equal to 
\[
1-\Phi(p_1,p_2,p_3)=1-\dfrac{\mu}{\epsilon}(p_1^2+p_2^2+p_3^2)=0.
\]
So the vector $p$ in \eqref{eq:vector p rotation of gradient S} satisfies the last equation and $p=\dfrac{1}{\mu}\nabla \psi$,  therefore $|\nabla \psi|^2=\epsilon\,\mu$.

\subsection{Case considered for the application of our results}\label{subsec:application of results when mu=a epsilon}
For the application of our results from Sections \ref{sec:vector snell law for anisotropic media}--\ref{sec:kappa less than one} we consider materials having permittivity and permeability tensors $\epsilon$ and $\mu$ that are positive definite symmetric constant matrices with $\mu=a\,\epsilon$, where $a$ is a positive number.  These are homogeneous materials that when $\epsilon$ is not the identity matrix are anisotropic.   
We will associate with such a material a norm as follows.
From the calculations above, the Fresnel equation in this case is as follows. 
From \eqref{eq:definition of tau} we get $\tau=\dfrac{1}{a}\,Id$, so
\[
\Phi(p)=a\,|p|^2,\qquad \Psi(p)=a^2\,|p|^4.
\]
Obviously, $\Phi^2=\Psi$ and so the Fresnel equation is $1-\Phi=0$, i.e., $|p|^2=1/a$ and therefore it has only one sheet. Then from \eqref{eq:vector p rotation of gradient S} the vector 
$
\dfrac{\mu^{1/2}}{\det \mu^{1/2}}\,\nabla \psi
$
satisfies the equation
\[
\left| \dfrac{\sqrt{a}\,\mu^{1/2}}{\det \mu^{1/2}}\,\nabla \psi\right|^2=1.
\]
The last expression induces the following dual norm
%\footnote{Notice that the factor $1/\sqrt{a}$ is introduced so that when $\mu$ is the identity and $\mu=a\,\epsilon$, i.e., $\epsilon=1/a\,Id$, the refractive index is $n=1/\sqrt{a}$ and the norm in this case is $n\,|x|$.}
\[
N^*(p)=\left| \dfrac{\sqrt{a}\,\mu^{1/2}}{\det \mu^{1/2}}\,p \right|.
\]
The norm $N^*$ is the dual to the norm given by 
\[
N(x)=\sup_{N^*(p)=1}|x\cdot p|=\dfrac{\det \mu^{1/2}}{\sqrt{a}}\,\left|\mu^{-1/2}\,x \right|,
\]
which is the norm we associate to the material.
Notice that if $\mu$ is the identity matrix, then $\epsilon=\dfrac{1}{a}Id$, the material is isotropic and has index of refraction $n=\sqrt{\epsilon\,\mu }=1/\sqrt{a}$. The norm obtained this way is then $N(x)=n\,|x|$, in agreement with the physical explanation for isotropic media given after \eqref{eq:vector snell law}. 

Now, if $N(x)=|A\,x|$ with $A$ a constant matrix, then $\nabla N(x)=\dfrac{1}{N(x)}\,A^tAx$. Therefore, having two materials $I$ and $II$ so that the wave fronts are given by norms $N_1(x)=|A_1x|$ and $N_2(x)=|A_2x|$, respectively, the Snell law \eqref{eq:vector snell law} takes the following form: 
Each incident ray traveling in medium $I$ with direction $x\in \Sigma_1$, i.e., $N_1(x)=1$,
with $x\cdot \nu\geq 0$ and striking the plane $\Pi$ at some point $P_0$ is refracted in medium $II$ into a direction $m\in \Sigma_2$, i.e., $N_2(m)=1$, if 
\begin{equation*}
A_2^tA_2 m-A_1^tA_1x\parallel \nu,
\end{equation*}
where $\nu$ is the unit normal at $P_0$ from medium $I$ to medium $II$.

In our application we have materials $I$ and $II$ having constant tensors $(\epsilon_1,a_1\,\epsilon_1)$ and $(\epsilon_2,a_2\,\epsilon_2)$, respectively, and therefore the associated norms to $I$ and $II$ are
\[
N_1(x)=\dfrac{\det \mu_1^{1/2}}{\sqrt{a_1}}\,\left|\mu_1^{-1/2}\,x \right|,\qquad
N_2(m)=\dfrac{\det \mu_2^{1/2}}{\sqrt{a_2}}\,\left|\mu_2^{-1/2}\,m \right|,
\]
respectively.
If we let 
\[
A_1=\dfrac{\det \mu_1^{1/2}}{\sqrt{a_1}}\,\mu_1^{-1/2},\qquad 
A_2=\dfrac{\det \mu_2^{1/2}}{\sqrt{a_2}}\,\mu_2^{-1/2},
\]
then $N_1(x)=\left|A_1\,x \right|$ and $N_2(m)=\left|A_2\,m \right|$.
To apply the results of the previous sections to this case, from the definition of $\kappa$ in \eqref{eq:kappa<1} we have
\[
\kappa=\sup_{\left|A_1\,x \right|=1}\left|A_2\,x \right|=\sup_{|z|=1}\left| A_2\,A_1^{-1}z \right|
=
\left\|A_2\,A_1^{-1} \right\|
\]
the norm of the matrix $A_2\,A_1^{-1}$ induced by the standard Euclidean norm $|\cdot |$.
\footnote{That is, $\kappa$ is the spectral norm of the matrix $A:=A_2\,A_1^{-1}$, i.e., $$\kappa=\sqrt{\text{maximum eigenvalue of $(A^t\,A)$}}$$.}
Hence when $\left\|A_2\,A_1^{-1} \right\|<1$ the results from Section \ref{sec:kappa less than one} are applicable to this case.
On the other hand, when $\kappa$ is defined by \eqref{eq:kappa>1}, setting $A=A_2A_1^{-1}$ we get 
\[
\sqrt{\text{minimum eigenvalue of $A^tA$}}=\inf_{|z|=1}\left| A_2\,A_1^{-1}z \right|>1,
\]
and the results from Remark \ref{rmk:kappa bigger than one} are applicable in this case.
We can also write
\[
A_2\,A_1^{-1}
=
\dfrac{\det \mu_2^{1/2}}{\sqrt{a_2}}\, \mu_2^{-1/2}\,
\dfrac{\sqrt{a_1}}{\det \mu_1^{1/2}}\, \mu_1^{1/2}
=
\sqrt{\dfrac{a_1}{a_2}}\,\dfrac{\det \mu_2^{1/2}}{\det \mu_1^{1/2}}\,\mu_2^{-1/2}\,\mu_1^{1/2}. \]
Once again notice that if $\mu_i$ is the identity matrix, then $\epsilon_i=\dfrac{1}{a_i}Id$, the materials are isotropic and have index of refraction $n_i=\sqrt{\epsilon_i\,\mu_i}=1/\sqrt{a_i}$. The norms are then $N_i(x)=n_i\,|x|$, $i=1,2$, and $\kappa=\sqrt{a_1/a_2}=n_2/n_1$ in agreement with the physical explanation for isotropic media given after 
\eqref{eq:vector snell law}. 
%THIS IS STRANGE BECAUSE DOES NOT AGREE WITH THE NORMS BEFORE. 

Finally, we remark that for the materials considered light rays travel in straight lines and they do not exhibit bi refringence, that is, each incident ray is refracted into only one ray.
The last property is because the Fresnel equation has only one sheet. 
That rays travel in straight lines follows from Fermat's principle of least time explained in Section \ref{sec:vector snell law for anisotropic media}. 
Indeed, let $X,Y$ be two points in space, $\gamma(\theta)=(1-\theta)\,X+\theta\,Y$, $0\leq \theta \leq 1$, and let $\phi(\theta)$ be any curve from $X$ to $Y$.
Then the optical length $T$ for each curve satisfies $T(\gamma)=\int_0^1 \|\gamma'(\theta)\|\,d\theta =\|X-Y\|$, and $T(\phi)=\int_0^1 \|\phi'(\theta)\|\,d\theta \geq \left\|\int_0^1 \phi'(\theta)\,d\theta\right\|=\|X-Y\|$.

%\begin{remark}[Final comments TO BE ADJUSTED]
For general anisotropic materials when $\epsilon$ is not a multiple of $\mu$, the Fresnel equation has two sheets, see Figure \ref{fig:fresnel surfaces}, and as mentioned before bi-refringence occurs.  
This is the case for crystals, that is, when $\epsilon$ is a diagonal constant matrix and $\mu=Id$.
%We hope to return to this problem in the future. 
%However, in this case rays also travel in straight lines.  

%We need to say something about the behavior of rays in crystals where bi-refringence occurs because the Fresnel equation has two sheets, one convex and another not convex. Rays travel also in straight lines. Perhaps say that one can deal with the convex surface in the same way as here. But a difficulty is to deal with the extraordinary rays....etc. Also a difficulty when the permittivity and permeability tensors depend on position is that rays have curved trajectories. To be written... 
%\end{remark} 

\color{black}

\section{Connection with optimal mass transport}\label{sec:optimal mass transport}

The setting up, analysis, and results from the previous sections allow us to cast the refraction problem in optimal transport terms.
However, the method used in Section \ref{sec:kappa less than one} to prove existence of solutions relies more on a deeper insight of the physical and geometric features of the refractor problem. 

To apply the optimal mass transport approach, we use the abstract set up in \cite[Section 3.2 and 3.3]{gutierrez-huang:farfieldrefractor} 
and from Definition \ref{def:refractor} introduce the cost function
\[
c(x,m)=\log \(\dfrac{1}{1-x\cdot p_2(m)}\)
\]
for $x\in \Omega_1\subset \Sigma_1$, $m\in \Omega_2\subset \Sigma_2$ and $\kappa<1$ with $\kappa$ in \eqref{eq:kappa<1}.
With \cite[Definition 3.9]{gutierrez-huang:farfieldrefractor} of $c$-concavity, we have that 
$\mathcal S=\{\rho(x)x:x\in \Omega_1\}$ is a refractor in the sense of Definition \ref{def:refractor} above if and only if $\log \rho$ is $c$-concave.
From the definition of $c$-normal mapping $\mathcal N_{c,\phi}$ given in \cite[Definition 3.10]{gutierrez-huang:farfieldrefractor}, and the definition of refractor mapping $\mathcal R_S$ given by \eqref{def:refractor map}, we have that $\mathcal R_S=\mathcal N_{c,\log \rho}$.
One can easily check that $\mathcal S$ is a weak solution of the refractor problem if and only if $\log \rho$ is $c$-concave and $\mathcal N_{c,\log \rho}$ is a measure preserving map in the sense of \cite[Equation (3.9)]{gutierrez-huang:farfieldrefractor} from $f(x)\,dx$ to $\mu$.
Hence existence and uniqueness up to dilations of the refractor problem follows as in  
\cite[Theorem 3.15]{gutierrez-huang:farfieldrefractor}.

From Remark \ref{rmk:kappa bigger than one} and using the cost function $c(x,m)=\log (x\cdot p_2(m)-1)$ we obtain similar results for $\kappa>1$, defined by \eqref{eq:kappa>1}.

%\bibliography{monamp}
%\bibliographystyle{alpha}

\end{document}